\documentclass[11pt,reqno]{amsproc}
\usepackage[english]{babel}
\usepackage{ucs}
\usepackage{amssymb,amsmath}
\usepackage{indentfirst}
\usepackage{amsthm}
\usepackage{amsfonts}
\numberwithin{equation}{section}
\theoremstyle{plain}
\newtheorem*{acknowledgements}{Acknowledgements}
\newtheorem{theorem}{\bf Theorem}

\newtheorem{prop}{\bf Proposition}
\newtheorem{oldassertion}{\bf Theorem}

\newtheorem{oldlemma}{\bf Lemma}

\usepackage{xcolor}
\usepackage{hyperref}

\hypersetup{pdfstartview=FitH, linkcolor=linkcolor,urlcolor=urlcolor, colorlinks=true, linkcolor=blue, citecolor=blue}

\begin{document}

\title{JOHN-NIRENBERG INEQUALITY FOR RIEMANN TYPE SERIES}
\author{Kristina Oganesyan}
\address{Lomonosov Moscow State University, Moscow Center for fundamental and applied mathematics, Centre de Recerca Matem\`atica, Universitat Aut\`onoma de Barcelona}
\email{oganchris@gmail.com}
\thanks{The work was supported by the Moebius Contest Foundation for Young Scientists and the Foundation for the Advancement of Theoretical Physics and Mathematics “BASIS” (grant no 19-8-2-28-1).}
\date{}

\selectlanguage{english}

\begin{abstract} 
We obtain an improvement of the John-Nirenberg inequality for the series of the form $\sum_{n=1}^{\infty}n^{-1}e^{2\pi i n^k x},\;k>2,$ on intervals consisting of points of a same convergent of their continued fractions. We also establish a convergence criterion for these series. 
\end{abstract}
	\maketitle

\section{Introduction}
We consider Riemann type series, i.e. the series 
\begin{align}\label{1}
    F_k:=\sum_{n=1}^{\infty}\frac{e(n^k x)}{n}:=\sum_{n=1}^{\infty}\frac{e^{2\pi i n^k x}}{n}
\end{align}
for $k\in\mathbb{N}$.

Studying of the series of the form $\sum_{n=1}^{\infty}n^{-s}e(n^k x)$ began with the so-called $``$Riemann's example$"$. Riemann is reported to assume that the imaginary part of the series for $k=s=2$ represents a continuous nowhere differentiable function (see, for example, \cite{G}). Although it was proved that the function is not differentiable at any irrational point and at a family of rational ones but is differentiable at an infinite number of rationals \cite{H,G}, it turned out that the Riemann's example possesses some curious and important properties. In particular, it is a multifractal function \cite{J}. In \cite{ChU}, the multifractal nature was proved for the series
$\sum_{n=1}^{\infty}n^{-s}e(P_k(n) x)$, where $P_k\in\mathbb{Z}[x],\;\deg P_k=k,\;1+k/2<s<k$.

In the case $s=1$, which is the one we will deal with, there is no local boundness, therefore, the analytic methods mentioned before do not work. Nevertheless, it is crucial to have an idea of how far such series are from being bounded.

It is easy to see, using the Abel transformation, that series \eqref{1} is divergent at a rational point $a/q$ if and only if $\sum_{n=1}^q e(\frac{a}{q}n^k)\neq 0$. In particular, the series diverges at any point $a/p^k,\;p\in\mathbb{P},\;\gcd(a,p)=\gcd(k,p)=1$, since $\sum_{n=1}^q e(\frac{a}{p^k}n^k)=p^{k-1}$. As for irrational points, we show that there holds the following theorem, which represents a generalization of \cite[Th. 2]{ChCU}.

\begin{theorem}\label{thm2} Let $1<k\in\mathbb{N}, \;x\in[0,1)\setminus\mathbb{Q}$ and $\{p_j/q_j\}$ be the convergents corresponding to the continued fraction of the number $x$. Denote $$\xi_{p/q}=\xi_{p/q}^k:=\sum_{t=0}^{q-1}e\big(t^k\frac{p}{q}\big).$$
Then series \eqref{1} converges if and only if
\begin{align*}
\sum_{i=1}^{\infty}\frac{\xi_{p_i/q_i}}{k q_i}\ln \frac{q_{i+1}}{q_i}
\end{align*}
does, and the difference between this sum and $F_k$ is uniformly bounded in $x$. 
\end{theorem}

However, Riemann type series, despite being divergent at a dense set of points, are known to belong to the space BMO$(\mathbb{T})$, i.e. the space of $1$-periodic functions $h$ such that
\begin{align*}
\|h\|_*:=\sup_{I\subset \mathbb{T}}{\|h\|_I}<\infty,\qquad \|h\|_I=|I|^{-1}\int_I|h-h_I|,
\end{align*}
where $I$ stands for an interval and $h_I$, for the average of $h$ on $I$. Moreover, it can be seen from \cite[Th. 1]{ChCU} that BMO-norms of $F_k$ are uniformly bounded in $k$ and for any sequence $\{b_k\}_{k=1}^{\infty}$ such that $|b_k|\leq 1$ for all $k$ BMO-norm of the function $\sum_{n=1}^{\infty}b_k n^{-1} e(n^k x)$ is bounded by the same constant. It also follows from \cite[Prop. 1]{ChCU} that there holds the inequality
\begin{align*}
    \limsup_{|I|\to 0}\bigg\|\sum_{n=1}^{\infty}\frac{b_k e(n^k x)}{n}\bigg\|_I\leq \frac{3(12\pi)^{1/3}}{k}.
\end{align*}

It is worth mentioning that the partial sums of the series $F_k$ do not converge to its sum in the metric of BMO. This can be seen from the following
\begin{prop}\label{new}
There exists a constant $C>0$ such that for any $m,k\in\mathbb{Z}^+,\;k>1,$ there holds
\begin{align*}
    \Big\|\sum_{n=m+1}^{\infty}\frac{e(n^kx)}{n}\Big\|_*\geq \frac{C}{k}.
\end{align*}
\end{prop}

One of the fundamental results for functions from BMO is the John-Nirenberg inequality \cite{JN} (see also \cite{Ga}).
\begin{oldassertion}(John, Nirenberg, 1961) 
There exist constants $C_1$ and $C_2$ such that for any function $h\in \text{BMO}$ on any interval $I$ there holds
\begin{align*}
|\{x\in I : |h(x)-h_I|>\lambda\}|\leq C_1|I|e^{-\frac{C_2\lambda}{\|h\|_*}}.
\end{align*}
\end{oldassertion}
This inequality shows how much the function differs from bounded ones. Moreover, it implies that a function from BMO belongs to $L_{loc}^p$ for any $p\in(1,\infty)$, which means in our periodic case that $F_k\in L^p(\mathbb{T})$ for any natural $k$ and any $p\in(1,\infty)$. 

For the case $k=2$ the John-Nirenberg inequality was improved in \cite{ChCU} as follows.
\begin{oldassertion}(Chamizo, C\'ordoba, Ubis, 2018)\label{chamizo} There exists a constant $C$ such that for any interval $J\subset I=I_{p/q}$ consisting of all the points having $p/q$ as their convergent, for any $\lambda>0$, there holds
\begin{align}\label{cham}
|\{x\in J:|F_2(x)-(F_2)_J|>\lambda\}|\leq C|J|e^{-\lambda\sqrt{2q}}.
\end{align}
\end{oldassertion}
This result turns out to give the sharp order estimate for the intervals $J=I_{p/q}$ with $q$ divisible by $4$. This optimality is due to the fact that quadratic Gauss sums can be evaluated explicitly. For $k>2$, Gauss sums behave in a more complicated way, however, we prove the following analogue of Theorem \ref{chamizo}.

\begin{theorem}\label{t}
There exists a constant $C$ such that for $k\geq 3$ for any interval $J\subset I=I_{p/q}$ and for any $\lambda>0$ 
\begin{align}\label{tt}
|\{x\in J:|F_k(x)-(F_k)_J|>\lambda\}|\leq |J|e^{-\frac{k\lambda}{A(k)}q^{\frac{1}{k}}+Cq^{\frac{1}{k}-\frac{1}{2^k(k-1)}}\ln q},
\end{align}
\end{theorem}
where
$$A(k):=\sup_{q\geq 2}\;\max_{\gcd \;(a,q)=1}\Bigg|\sum_{x=0}^{q-1}e\left(\frac{ax^k}{q}\right)\Bigg|q^{\frac{1}{k}-1}.$$ 
According to \cite{BS}, $A(k)\leq 4.709236...$. It is worth mentioning that $A(2)=\sqrt{2}$, therefore, for $k=2$ we find $kA(k)^{-1}q^{\frac{1}{k}}=\sqrt{2q}$, that is, the main term in the exponent in \eqref{tt} coincides with the one in \eqref{cham} for $k=2$.

To prove Theorem \ref{chamizo}, an important property of the power $k=2$ was significantly used, namely, the fact that for all the points $x$ of the interval $I_{p/q}$ there holds $|x-p/q|<q^{-k}$ for $k=2$, but for $k>2$ the inequality can be true only on a set of a small measure. Thus, the value of $F_2(x)$ on the whole interval $I$ in a certain way is sufficiently close to that of $F_2(p/q)$, which is no longer true for $F_k,\;k>2$. 

In order to obtain suitable estimates for $k>2$ at points that are not sufficiently close to $p/q$, we use the estimates \cite{O} of Weyl sums, i.e. sums of the form $\sum_{k=1}^P e^{2\pi ih(k)}$, depending on rational approximations of the leading coefficient of a polynomial $h\in\mathbb{R}[x]$.

\section{Proofs of the results}

We first formulate the results \cite[Cor. 1-3, Rem. 7]{O} that we use to estimate Weyl sums.

\begin{oldlemma}\label{old} Let $0<y\in\mathbb{R}\setminus \mathbb{Q},\;\varepsilon\in(0,1),\;P\in\mathbb{N},\;3\leq n\in\mathbb{N}$ and let $f$ be a real monic polynomial of degree $n$. For numbers $C$ and $M$ define $\beta=\beta(y,C,M):=y-C/M$.

(a) If there is no pair of coprime numbers $C$ and $M\leq P^{\varepsilon}$ such that $|\beta|\leq P^{\varepsilon-1},$ then 
$$\bigg|\sum_{k=1}^P e^{\frac{2\pi if(k)y}{n!}}\bigg|\leq C_1 P^{1-\frac{\varepsilon}{2^n}}.$$

(b) If there exists a pair of coprime numbers $C$ and $M\leq P^{\varepsilon}$ such that
$P^{\varepsilon-n}< |\beta|\leq P^{\varepsilon-1},$ then 
\begin{align}\label{otbr}
\bigg|\sum_{k=1}^P e^{\frac{2\pi if(k) y}{n!}}\bigg|\leq C_2 \Big(P^{1-\frac{\varepsilon}{2^{n-1}}}+P^{1-\frac{n-\varepsilon}{2^{n-1}(n-1)}}\beta^{-\frac{1}{2^{n-1}(n-1)}}M^{-\frac{1}{2^n(n-1)}}\Big).
\end{align}

(c) If there exists a pair of coprime numbers $C$ and $M\leq P^{\varepsilon}$ such that 
$|\beta|\leq P^{\varepsilon-n},$ then
\begin{align*}
\bigg|\sum_{k=1}^P e^{\frac{2\pi if(k) y}{n!}}\bigg|\leq C_3\frac{P}{M^{\frac{1}{2^n(n-1)}}}.
\end{align*}
The constants $C_1,\;C_2$ depend on $n$ and $\varepsilon$, and the constant $C_3$, only on $n$.
\end{oldlemma}

Roughly speaking, Lemma \ref{old} shows that the worse rational numbers approximate the leading coefficient the better estimate of the Weyl sum we obtain.

\begin{proof}[\it Proof of Proposition \ref{new}] It was proved by Fefferman in his unpublished work, using duality between the Hardy space $H^1$ and BMO, that for a series $\sum_{n=1}^{\infty}a_n e(nx),\;a_n\geq 0,$ the following expression is equivalent to its BMO-norm (see \cite{B,SS}):
\begin{align*}
    S(\{a_n\}):=\bigg(\sup_{N\in\mathbb{N}}\sum_{j=1}^{\infty}\Big(\sum_{n=jN}^{(j+1)N-1}a_n\Big)^2\bigg)^{\frac{1}{2}}.
\end{align*}
Considering for any $m,k\in\mathbb{Z}^+,\;k>1,$ the sequence
\begin{equation*}
    a_l:=\begin{cases}
        n^{-1}, & if\; l=n^k-m,\\
        0, &  \text{otherwise},
            \end{cases}
\end{equation*}
we will show that $S(\{a_n\})\geq C/k$ for some constant $C$, thus the proposition will be proved.

Note that for a fixed $k$ and for $j+1\leq \frac{(N/2k)^{\frac{k}{k-1}}-m}{N}=:J(k,m,N)$ we have

\begin{align*}
    ((j+1)N+m)^{\frac{1}{k}}-(jN+m)^{\frac{1}{k}}-2> \frac{N}{k}((j+1)N+m)^{\frac{1-k}{k}}-2\geq 0,
\end{align*}
which yields
\begin{align}\label{yields}
   \lceil((j+1)N+m)^{\frac{1}{k}}\rceil-\lceil(jN+m)^{\frac{1}{k}}\rceil\geq \frac{((j+1)N+m)^{\frac{1}{k}}-(jN+m)^{\frac{1}{k}}}{2}.
\end{align}

Taking into account \eqref{yields}, we derive

\begin{align*}
    &\sum_{j=1}^{\infty}\Big(\sum_{n=jN}^{(j+1)N-1}a_n\Big)^2=\sum_{j=1}^{\infty}\bigg(\sum_{n=\lceil (jN+m)^{\frac{1}{k}}\rceil}^{\lceil ((j+1)N+m)^{\frac{1}{k}}\rceil-1}\frac{1}{n}\bigg)^2\\ &\geq\sum_{j=1}^{\infty}\bigg(\Big(\lceil((j+1)N+m)^{\frac{1}{k}}\rceil-\lceil(jN+m)^{\frac{1}{k}}\rceil\Big)\frac{1}{((j+1)N+m)^{\frac{1}{k}}}\bigg)^2\\
    &\geq \frac{1}{4}\sum_{j=1}^{J(k,m,N)/2}\bigg(1-\Big(1-\frac{1}{j+1+\frac{m}{N}}\Big)^{\frac{1}{k}}\bigg)^2\geq \frac{1}{8k^2}\int\limits_{1}^{J(k,m,N)/2}\big(x+\frac{m}{n}\big)^{-2} dx \\
    &=\frac{1}{8k^2}\bigg(\frac{1}{1+\frac{m}{N}}-\frac{2}{J(k,m,N)}\bigg)\to\frac{1}{8k^2},
\end{align*}
as $N\to\infty$, which completes the proof.
\end{proof}

Turn now to the questions of representations of $F_k$ at an irrational point. Theorem \ref{thm2} follows immediately from

\begin{prop}\label{prop} Let $1<k\in\mathbb{N},\; x\in[0,1)\setminus\mathbb{Q}$ and $\{p_j/q_j\}$ be the convergents corresponding to the continued fraction of the number $x$. Then for any $\tau\geq 2$, for $q_i^{\tau}\leq m< q_{i+1}^{\tau}$, the following equality holds
\begin{align*}
\sum_{m\leq n <q_{i+1}^{\tau} }\frac{1}{n(n+1)}\sum_{l=1}^n e(l^k x)= \frac{\xi_{p_i/q_i}}{k q_i}\ln^+ \frac{q_iq_{i+1}}{m^k}+O\Big(q_i^{-\frac{1}{2^{k+1}}}+q_i^{-\frac{1}{2^{k}(k-1)}}\ln q_i\Big),
\end{align*}
where $\ln^+ y=\max\{\ln y, 0\}$.
\end{prop}

\begin{proof}[\it Proof.]
Let $\tau\geq 2$ and $\varepsilon\in (0,\frac{1}{3})$ be such that 
$$\max\{(1-2\varepsilon)^{-1},(2\varepsilon)^{-1}\}\leq\tau<\varepsilon^{-1}$$ (for example, we can put $\varepsilon:=(2\tau)^{-1}$).
Assume that for some $n$ in the interval $q_i^{\tau}\leq n < q_{i+1}^{\tau}$ there exist coprime numbers $C$ and $M\leq n^{\varepsilon}$ such that
\begin{align}\label{u}
\Big|k!x-\frac{C}{M}\Big|\leq n^{\varepsilon-1}.
\end{align}
Then we have
$$|\beta|:=\Big|x-\frac{C}{Mk!}\Big|\leq \frac{n^{\varepsilon-1}}{k!}<\frac{1}{2(Mk!)^2},$$
since $n^{\varepsilon-1}=n^{3\varepsilon-1-2\varepsilon} \leq n^{3\varepsilon-1} M^{-2}<\frac{1}{2M^2k!}$ for sufficiently large $q_i$. Thus, by \cite[Th. 19]{Kh} $C/Mk!=p_j/q_j$ is a convergent of $x$ with $M\leq q_j\leq Mk!$. So, according to \cite[Th. 13]{Kh}, we infer that
$$\frac{1}{2Mk!q_{j+1}}\leq\frac{1}{2q_jq_{j+1}}<|\beta|\leq \frac{n^{\varepsilon-1}}{k!},$$
which gives for sufficiently large $q_i$
\begin{align*}
q_{j+1}^{\tau}>q_{j+1}\geq \bigg(\frac{n^{1-\varepsilon}}{2M}\bigg)^{\tau}\geq \bigg(\frac{n^{1-2\varepsilon}}{2}\bigg)^{\tau}\geq n>n^{1-\tau\varepsilon}\left(\frac{q_j}{k!}\right)^{\tau}>q_j^{\tau},
\end{align*}
whence $i=j$.

Let $N:=\lfloor |\beta|^{-\frac{1}{k}}\rfloor$. If we have $n\in[m,q_{i+1}^{\tau})$ such that $|\beta|\leq n^{-k}$, then $n\leq N$ and there holds
\begin{align}\label{appr}
\left|\sum_{l=1}^n e(l^k x)-\frac{\xi_{p_i/q_i}}{q_i}\sum_{l=1}^ne(l^k\beta)\right|&\leq \sum_{l=1}^{n-1}\Big |e(l^k\beta)-e\big((l+1)^k\beta\big)\Big | \bigg |\sum_{t=1}^l e\Big(\frac{p_i}{q_i}t^k\Big)-\frac{l\xi_{p_i/q_i}}{q_i}\bigg|\nonumber\\
+\big |e(n^k\beta)\big | \bigg |\sum_{t=1}^n e\Big(\frac{p_i}{q_i}t^k\Big)&-\frac{n\xi_{p_i/q_i}}{q_i}\bigg|\leq q_i\sum_{l=1}^{n-1} 4\pi k(l+1)^{k-1}|\beta|+q_i\nonumber\\
&\quad\quad\quad\quad\quad<8q_i\leq 8n^{\frac{1}{\tau}}\leq 8n^{\frac{1}{2}}.
\end{align}
At the same time

\begin{align}\label{int}
\sum_{m\leq n \leq N}\frac{1}{n(n+1)}\sum_{l=1}^ne(l^k\beta)=\sum_{m\leq n \leq N}\frac{e(n^k\beta)}{n}+O(1)=\int\limits_{m}^{N}\frac{e(y^k\beta)}{y}dy+O(1),
\end{align}
since for $m\leq y\leq N,c\in[0,1]$
\begin{align*}
\left|\frac{e((y-c)^k\beta)}{y-c}-\frac{e(y^k\beta)}{y}\right|\leq \frac{2\pi k y^{k-1}|\beta|}{y}+\frac{c}{y(y-c)}<\frac{3\pi k}{y^2}.
\end{align*}
The integral in \eqref{int} is equal to $\big(\text{Ei}(2\pi i|\beta|N^k)-\text{Ei}(2\pi i|\beta|m^k)\big)/k$ for $\beta>0$, otherwise, to its conjugate. Taking into account that for $y>0$  $$\text{Ei}(iy)=\min\{0,\ln y\}+O(1)$$ and that $2\pi|\beta|N^k>1$, we derive

\begin{align}\label{int2}
\int\limits_{m}^{N}\frac{e(y^k\beta)}{y}dy=\frac{1}{k}\ln^+|\beta|^{-1}m^{-k}+O(1)=\frac{1}{k}\ln^+\frac{q_{i+1}q_{i}}{m^k}+O(1).
\end{align}
Here we used the inequalities $(2q_iq_{i+1})^{-1}<|\beta|< (q_iq_{i+1})^{-1}$, the latter one is true due to \cite[Th. 9]{Kh}.

Inequality \eqref{u} can be true also for $n\in[m,q_{i+1}^{\tau})$ from the following intervals: $|\beta|^{-\frac{1}{k}}< n\leq |\beta|^{-\frac{1}{k-\delta}},\; |\beta|^{-\frac{1}{k-\delta}}< n\leq |\beta|^{-1},\;|\beta|^{-1}< n\leq |\beta|^{-\frac{1}{1-\varepsilon}}$, where $\delta:=k(1-\ln|\beta|/\ln M)^{-1}$. We denote these intervals respectively by $P_1,\;P_2$ and $P_3$. So, for $n\in P_2$, due to Lemma \ref{old}, (b), (where using the inequality
\begin{align*}
n^{\delta}=M^{\frac{k\ln n}{\ln M-\ln |\beta|}}\geq M^{\frac{k\ln n}{\ln n^{\varepsilon}+\ln n^{k-\delta}}}>M^{\frac{1}{2(k-1)}},
\end{align*} 
we discard the first term of \eqref{otbr}), we have

\begin{align*}
\bigg|\sum_{l=1}^n e(l^k x)\bigg|= O\Big(n^{1-\frac{k-\delta}{2^{k-1}(k-1)}} |\beta|^{-\frac{1}{2^{k-1}(k-1)}}M^{-\frac{1}{2^k(k-1)}}\Big).
\end{align*}
Hence

\begin{align}\label{2prom}
\sum_{n\in P_2}\frac{1}{n(n+1)}\bigg|\sum_{l=1}^n e(l^k x)\bigg|=O\Big( M^{-\frac{1}{2^k(k-1)}}\Big)=O\Big(q_i^{-\frac{1}{2^k(k-1)}}\Big).
\end{align}
Similarly, applying Lemma \ref{old}, (b), to $n\in P_3$ we obtain

\begin{align}\label{3prom}
\sum_{n\in P_3}\frac{1}{n(n+1)}\bigg|\sum_{l=1}^n e(l^k x)\bigg|= O\Big(q_i^{-\frac{1}{2^k(k-1)}}\Big).
\end{align}
Finally, according to Lemma \ref{old}, (c), using also the two-sided inequality $(2q_iq_{i+1})^{-1}<|\beta|< (q_iq_{i+1})^{-1}$, we see that

\begin{align}\label{1prom}
\sum_{n\in P_1}&\frac{1}{n(n+1)}\bigg|\sum_{l=1}^n e(l^k x)\bigg|\leq C\ln\frac{(2q_iq_{i+1})^{\frac{1}{k-\delta}}}{(q_iq_{i+1})^{\frac{1}{k}}} M^{-\frac{1}{2^k(k-1)}} \nonumber\\
&\leq C'\ln(q_iq_{i+1})^{\delta}M^{-\frac{1}{2^k(k-1)}}\leq C'\ln(q_iq_{i+1})^{\frac{k\ln M}{\ln M-\ln|\beta|}}M^{-\frac{1}{2^k(k-1)}}\nonumber\\
&\qquad\leq C'kM^{-\frac{1}{2^k(k-1)}}\ln M =O\Big(q_i^{-\frac{1}{2^k(k-1)}}\ln q_i\Big).
\end{align}
Now, the only case we did not consider yet is that of the numbers $n\in[m,q_{i+1}^{\tau})$ not satisfying \eqref{u}, the set of such numbers $n$ we denote by $U$. We use Lemma \ref{old}, (a),
to obtain

\begin{align}\label{3part}
\sum_{n\in U}\frac{1}{n(n+1)}\bigg|\sum_{l=1}^n e(l^k x)\bigg|= O \Big( (q_i^{\tau})^{-\frac{\varepsilon}{2^k}}\Big)=O\Big(q_i^{-\frac{1}{2^{k+1}}}\Big).
\end{align}

Summing up \eqref{appr} -- \eqref{3part}, we arrive at the needed equality.
\end{proof}

From now on we fix some $k>2$ and use the notation $F(x):=F_k(x)$.

\begin{proof}[\it Proof of Theorem \ref{t}.]
First we prove the theorem for $J=I$.

Let $F^-(x)=\sum_{n<q^{\tau}}\frac{1}{n(n+1)}\sum_{l=1}^n e(l^k x),\; F^+(x)=F(x)-F^-(x).$ Then for some point $x'\in I$
\begin{align}\label{F-}
&|F^-(x)-F^-_I|\leq \bigg|\sum_{q^{\frac{1}{k-\tau^{-1}}}\leq n<q^{\tau}}\frac{1}{n(n+1)}\sum_{l=1}^n e(l^k x)\bigg|\nonumber\\
&\qquad+\bigg|\sum_{q^{\frac{1}{k-\tau^{-1}}}\leq n<q^{\tau}}\frac{1}{n(n+1)}\sum_{l=1}^n e(l^k x')\bigg|+2\pi k|I|\sum_{n<q^{\frac{1}{k-\tau^{-1}}}}n^{k-2}\cdot n\nonumber\\
&\qquad\qquad\leq 2\max_{y\in I}\bigg|\sum_{q^{\frac{1}{k-\tau^{-1}}}\leq n<q^{\tau}}\frac{1}{n(n+1)}\sum_{l=1}^n e(l^k y)\bigg|+ O\Big(q^{-2}(q^{\frac{1}{k-\tau^{-1}}})^{k}\Big).
\end{align}
Due to \cite[Note for Th. 14]{korobov}, for $\varepsilon_1:=\tau^{-1}$, for $q^{\frac{1}{k-\tau^{-1}}}\leq n\leq q^{\tau}$ and $y\in I$, we have

\begin{align*}
\Big|\sum_{l=1}^n e(l^k y)\Big|\leq n^{1-\frac{\varepsilon_1}{2^k}}=n^{1-\frac{\varepsilon}{2^{k-1}}},
\end{align*}
whence
\begin{align}\label{F-}
|F^-(x)-F^-_I|=O\Big(q^{-\tau\varepsilon 2^{-k+1}}\Big)=O\Big(q^{-2^{-k}}\Big).
\end{align}
Likewise as in the proof of \cite[Th. 3]{ChCU},

\begin{align*}
F_I^+&=O\bigg(q_i^{-\frac{1}{2^{k+1}}}+q_i^{-\frac{1}{2^k(k-1)}}\ln q_i+\sum_{q_j\geq q}|I|^{-1}\int\limits_I q_j^{-\frac{1}{k}}(y)\ln a_{j+1}(y)dy\bigg)\nonumber\\
&\qquad\qquad\qquad=O\Big(q^{-\frac{1}{2^{k+1}}}+q^{-\frac{1}{2^k(k-1)}}\ln q\Big).
\end{align*}
Thus, using Proposition \ref{prop}, we infer that
\begin{align}\label{from_prop}
F(x)-F_I= \frac{1}{k}\sum_{q_j\geq q}\frac{\xi_{p_j/q_j}}{q_j}\ln ^+\frac{q_{j+1}}{q_j}+O\Big(q^{-\frac{1}{2^{k+1}}}+q^{-\frac{1}{2^k(k-1)}}\ln q\Big).
\end{align}
Here by \cite[Th. 1]{BS} we have
\begin{align}\label{Shp}
\Big| \frac{\xi_{p_j/q_j}}{q_j}\Big|\leq A(k) q_j^{-\frac{1}{k}}.
\end{align}

Now we follow the proof of \cite[Th. 3]{ChCU} but making some changes. For $b=(b_1,b_2,...,b_d)\in \mathbb{N}^d$, where $d>2$ is to be defined later, if $p/q=[0;a_1,...,a_{j_0}]$, then we denote by $I_b$ the interval $I_{p'/q'}$ corresponding to $p'/q'=[0;a_1,...,a_{j_0},b_1,...,b_d]$. Thus, $I=\cup_b I_b$. Taking into account that $q_{j+1}/q_j\leq a_{j+1}+1$, we obtain from \eqref{from_prop} and \eqref{Shp}

\begin{align}\label{a}
\frac{|F(x)-F_I|k}{A(k)}\leq \sum_{l=1}^dq_{j_0+l-1}^{-\frac{1}{k}}\ln b_l&+\sum_{n=1}^{\infty}q_{j_0+d+n-1}^{-\frac{1}{k}}\ln a_{j_0+d+n}\nonumber\\
&+O\Big(q^{-\frac{1}{2^{k+1}}}+q^{-\frac{1}{2^k(k-1)}}\ln q\Big).
\end{align}

Likewise as in the proof of \cite[Th. 3]{ChCU}, the measure of the set of points $x\in I_b$ such that $a_{j_0+d+n}(x)\leq Cn^2 e^{\lambda kA(k)^{-1}q^{\frac{1}{k}}}$ for all $n\in\mathbb{N}$ differs from $|I_b|$ in $O(|I_b|e^{-\lambda kA(k)^{-1}q^{\frac{1}{k}}})$ due to \cite[(18)]{ChCU}. Therefore, from now on we assume that these inequalities hold, so the second sum in \eqref{a} is $O(\lambda(q/q_{j_0+d})^{\frac{1}{k}})$. Similarly,

\begin{align}\label{again}
|\{x\in I:|F(x)-F_I|>\lambda\}|\leq |I|\sum_{b\in \mathfrak{B}_{\lambda}}(b_1...b_d)^{-2},
\end{align}
where $\mathfrak{B}_{\lambda}$ stands for the set of $b$ such that 
\begin{align}\label{b}
\frac{A(k)}{k}\sum_{l=1}^dq_{j_0+l-1}^{-\frac{1}{k}}\ln b_l+C\lambda\Big(\frac{q}{q_{j_0+d}}\Big)^{\frac{1}{k}}+Cq^{-\frac{1}{2^{k+1}}}+Cq^{-\frac{1}{2^k(k-1)}}\ln q>\lambda
\end{align}
for some constant $C$. Choosing $d$ so that there holds $2^{d/2k-1}>C$, we derive that the second term in \eqref{b} does not exceed $\lambda/2$. It is clear that for $k>2$ we can assume that $\lambda\big(q^{-\frac{1}{2^{k+1}}}+q^{-\frac{1}{2^k(k-1)}}\ln q\big)$ is greater than some fixed constant. Thus, the sum of the third and the fourth terms in \eqref{b} does not exceed $\lambda/4$. This means that we can find an index $l_0,\;1\leq l_0\leq d$ such that
\begin{align*}
\ln b_{l_0}\geq \frac{\lambda}{4d}\frac{k}{A(k)}q^{\frac{1}{k}},
\end{align*}
whence

\begin{align*}
\frac{q_{j_0+d}}{q}\geq b_{l_0}\geq e^{\frac{k\lambda}{4dA(k)}q^{\frac{1}{k}}}\geq \Big(\frac{k\lambda}{4dA(k)}q^{\frac{1}{k}}\Big)^k.
\end{align*} 
Hence, if $k>2$, \eqref{b} becomes
\begin{align*}
\frac{A(k)}{k}\sum_{l=1}^d\Big(\frac{q}{q_{j_0+l-1}}\Big)^{\frac{1}{k}}\ln b_l>\lambda q^{\frac{1}{k}}-3Cq^{\frac{1}{k}-\frac{1}{2^k(k-1)}}\ln q.
\end{align*}
Since $q/q_{j_0+l-1}<1$ for $l>1$, we can apply \cite[L. 2]{ChCU} and deduce, using \eqref{again},
\begin{align}\label{c}
|\{x\in I:|F(x)-F_I|>\lambda\}|&\leq |I|\Big(e^{\lambda q^{\frac{1}{k}}-3Cq^{\frac{1}{k}-\frac{1}{2^k(k-1)}}\ln q}\Big)^{-\frac{k}{A(k)}}\nonumber\\
&\qquad\qquad\leq|I|e^{-\frac{k\lambda}{A(k)}q^{\frac{1}{k}}+C'q^{\frac{1}{k}-\frac{1}{2^k(k-1)}}\ln q}.
\end{align}

Let us turn to an arbitrary interval $J\subset I$. For this case our proof is almost the same as the one of \cite[Prop. 3]{ChCU} (we keep using its notations substituting only $I$ with $J$), we just need to outline some modifications. Firstly, \cite[(23)]{ChCU} becomes

\begin{align*}
&|\{x\in J:|F(x)-F_J|>\lambda\}|\\
&\qquad\qquad=O\left(\sum_{b=c}^{d-1}|J^b|\min\big\{1,e^{k A(k)^{-1}(|F_{J^b}-F_J|-\lambda)q^{\frac{1}{k}}+C'q^{\frac{1}{k}-\frac{1}{2^k(k-1)}}\ln q}\big\}\right).
\end{align*}
Secondly, estimate \eqref{F-} remains true if we substitute $I$ with any interval containing in $I$, therefore, using Proposition \ref{prop},

\begin{align*}
&F_{J^b}-F_J=F_{J^b}^+-F_J^++O\Big(q^{\frac{1}{2^k(k-1)}}\ln q\Big)=\frac{\xi_{p/q}}{kq}\Big(\ln b-\big(\ln a_{j_0+1}(x)\big)_J\Big)\\
&\qquad+O\bigg(\sum_{q_j>q}\frac{1}{|J|}\int\limits_Jq_j^{-\frac{1}{k}}(y)\ln a_{j+1}(y)dy\bigg)+O\bigg(\sum_{q_j>q}\frac{1}{|J^b|}\int\limits_{J^b}q_j^{-\frac{1}{k}}(y)\ln a_{j+1}(y)dy\bigg)\\
&\qquad\qquad+O\left(q^{\frac{1}{2^k(k-1)}}\ln q\right)=\frac{\xi_{p/q}}{kq}\Big(\ln b-\big(\ln a_{j_0+1}(x)\big)_J\Big)+O\Big(q^{\frac{1}{2^k(k-1)}}\ln q\Big).
\end{align*}
Applying \cite[Th. 1]{BS} and the fact that $(\ln(a_{j_0+1}/c))_J=O(1)$ we have
\begin{align*}
|F_{J^b}-F_J|\leq \frac{A(k)q^{-\frac{1}{k}}}{k}\ln \frac{b}{c}+O\Big(q^{\frac{1}{2^k(k-1)}}\ln q\Big),
\end{align*}
which means that in our case $\mathfrak{B}$ is of a slightly different form:
\begin{align*}
\mathfrak{B}=\Big\{b\in[c,d-1]\cap \mathbb{Z}: \ln(b/c)\leq \lambda q^{\frac{1}{k}}kA(k)^{-1}-Cq^{\frac{1}{2^k(k-1)}}\ln q\Big\}.
\end{align*}
The rest of the proof is the same as that of \cite[Prop. 3]{ChCU}.

\end{proof}

\begin{acknowledgements}
I am grateful to Sergey Tikhonov for bringing attention to the problem in light of the recent results \cite{O} and for valuable suggestions during the preparation of the paper.
\end{acknowledgements}


\begin{thebibliography}{11}

\bibitem[1]{B}
Bergh J., {\it Functions of bounded mean oscillation and Hausdorff-Young type theorems}, Function Spaces and Applications, Springer Berlin Heidelberg, Berlin, Heidelberg, 1988, 130--136.

\bibitem[2]{BS}
Banks W., Shparlinski I., {\it On Gauss sums and the evaluation of Stechkin's constant}, Math. Comp., 85 (301)  (2016), 2569--2581.

\bibitem[3]{ChCU}
Chamizo F., C\'ordoba A., Ubis A., {\it Fourier series in BMO with number theoretical implications}, Math. Ann., 376  (2020), 457--473. 

\bibitem[4]{ChU}
Chamizo F., Ubis A., {\it Multifractal behavior of polynomial Fourier series}, Adv. Math., 250  (2014), 1--34.

\bibitem[5]{Ga}
Garnett J., {\it Bounded Analytic Functions}, Springer, New York, NY, 2007.

\bibitem[6]{G}
Gerver J., {\it The differentiability of the Riemann function at certain rational multiples of $\pi$}, Amer. J. Math., 92   (1970), 33--55.

\bibitem[7]{H}
Hardy G.H., {\it Weierstrass’s non-differentiable function}, Trans. Amer. Math. Soc., 17   (1916), 301--325.

\bibitem[8]{J}
Jaffard S., {\it The spectrum of singularities of Riemann’s function}, Rev. Mat. Iberoam., 12   (1996), 441--460.

\bibitem[9]{JN}
John F., Nirenberg L., {\it On functions of bounded mean oscillation}, Comm. Pure Appl. Math., 14  (1961), 415--426.

\bibitem[10]{Kh}
Khinchin A., {\it Continued fractions}, Dover Publications Inc., Mineola, NY, 1997. 

\bibitem[11]{korobov}
Korobov N., {\it Exponential sums and their applications}, Kluwer Academic Publishers, Dordrecht, Boston, 1992.

\bibitem[12]{O}
Oganesyan K.A., {\it Uniform convergence criterion for non-harmonic sine series}, Sb. Math., 212 (1)  (2021), 70--110.

\bibitem[13]{SS}
Sledd W. T., Stegenga D. A., {\it An $H^1$ multiplier theorem}, Ark. Mat. 19 (1--2)  (1981), 265--270.


\end{thebibliography}
\end{document}